\documentclass[12pt,twoside]{amsart}
\usepackage[latin1]{inputenc}
\usepackage{stmaryrd}
\usepackage{times}
\usepackage{amsthm}
\usepackage{amssymb, verbatim}
\usepackage{amsmath}
\usepackage{mathrsfs}

\topmargin = - 40 pt 
\newcommand     {\Cset}    {{\mathbb C}}
\newcommand     {\Nset}    {{\mathbb N}}
\newcommand     {\Zset}    {{\mathbb Z}}

\newcommand     {\llbr}     {{\llbracket}}
\newcommand     {\rrbr}     {{\rrbracket}}

\newcommand {\1}    {{\bf 1}}

\newcommand     {\CD}   {{\C[\partial]}}

\newcommand     {\C}    {{\Bbb C}}

\newcommand     {\Z}    {{\Bbb Z}}

\newcommand     {\End}  {\mbox{End}}

\newcommand     {\Tor}  {{\mbox{\it Tor }}}

\newcommand     {\id}   {\mbox{id}}

\newtheorem{thm}{Theorem}[section]

\newtheorem{lemma}{Lemma}[section]
\newtheorem{prop}{Proposition}[section]

\theoremstyle{definition}

\newtheorem{defn}{Definition}[section]

\newtheorem{ex}{Example}[section]

\newtheorem{rem}{Remark}[section]

\DeclareMathOperator{\gr}{gr} 
\DeclareMathOperator{\spn}{span}

\numberwithin{equation}{section}

\begin{document}

\title[Finiteness and orbifold VOAs]
{Finiteness and orbifold vertex operator algebras}
\author[A.~D'Andrea]{Alessandro D'Andrea}
\address{Dipartimento di Matematica, Universit\`a degli Studi di
Roma ``La Sapienza'', Roma}
\email{dandrea@mat.uniroma1.it}
\date{\today}

\maketitle

\tableofcontents

\section{Introduction}

It has been observed in many instances, see \cite{Linshaw} and references therein, that a strong finiteness condition on a (simple) vertex operator algebra, or VOA, is inherited by subalgebras of invariant elements under the action of a reductive (possibly finite) group of automorphisms. This amounts to a quantum version of Hilbert's basis theorem for finitely generated commutative algebras, but is typically dealt with, in the relevant examples, by means of invariant theory.

A big issue that needs to be addressed in all attempts towards proving the above statement in a general setting is its failure in the trivial commutative case. A commutative vertex algebra is nothing but a commutative differential algebra, and it has long been known that both the noetherianity claim contained in Hilbert's basis theorem, and the finiteness property of invariant subalgebras, cannot hold for differential commutative algebras. Counterexamples are easy to construct, and a great effort has been spent over the years into finding the appropriate generalization of differential noetherianity. Every investigation of finiteness of vertex algebras must first explain the role played by noncommutativity and its algebraic consequences.

In this paper, I announce some results in this direction, and claim that every strongly finitely generated simple vertex operator algebra satisfies the ascending chain condition on its right ideals. Here, a VOA is simple if it has no nontrivial quotient VOAs, whereas the right ideals involved in the ascending chain conditions are subspaces that are stable both under derivation and right multiplication with respect to the normally ordered product; even a simple VOA may have very many ideals of this sort, and they are better suited when addressing finiteness conditions. Right noetherianity of simple VOAs is the first algebraic property, as far as I know, that can be proved on a general level, and explains a first important difference between the commutative and noncommutative situation.

The paper is structured as follows: in Sections 2 and 3, I rephrase the vertex algebra structure in the context of left-symmetric algebras, and describe how the normally ordered product and the singular part of the Operator Product Expansion relate to each other. In Sections 4 and 5, I recall the concept of strong generators for a VOA, and explain its interaction with Li's filtration \cite{abelian}, and its generalization to structures that are weaker than proper VOAs. Section 6 explains the role of what I call {\em full ideals} into proving some version of noetherianity for a VOA. Speculations on how to use noetherianity in order to address strong finiteness of invariant subalgebras of a strongly finitely generated VOA are given in Section 7. I thankfully acknowledge Victor Kac for his suggestion that Lemma \ref{kac} might be useful in the study of finiteness of orbifold VOAs.

\section{What is a vertex operator algebra?}

\subsection{Left-symmetric algebras}

A {\em left-symmetric algebra} is a (complex) vector space $A$ endowed with a bilinear product $\cdot : A \otimes A \to A$ which is in general neither commutative nor associative. The associator $(a, b, c) = (ab)c - a(bc)$ must however satisfy the following left-symmetry axiom:
$$(a, b, c) = (b, a, c),$$
for every choice of $a, b, c \in A$. One may similarly define {\em right-symmetric algebras} by requiring that $(a, b, c) = (a, c, b)$. Clearly, an associative algebra is both left- and right-symmetric.

If $A$ is any (non-commutative, non-associative) algebra, the commutator $[a, b] = ab - ba$ satisfies
$$[a, [b, c]] = [[a, b], c] + [b, [a, c]] + (b, a, c) - (a, b, c) - (c, a, b) + (a, c, b) + (c, b, a) - (b, c, a),$$
for all $a, b, c \in A$. When $A$ is either left- or right-symmetric, this reduces to the ordinary Jacobi identity
$$[a, [b, c]] = [[a, b], c] + [b, [a, c]],$$
and the commutator thus defines a Lie bracket on $A$. In a left-symmetric algebra, commutativity implies associativity, as
\begin{equation}\label{commass}
(a, b, c) = [c, a]b + a[c, b] - [c, ab].
\end{equation}
A similar identity holds in the right-symmetric case.

\subsection{Differential graded left-symmetric algebras}

A {\em differential graded left-symmetric algebra} (henceforth, a {\em DGLsA}) is a non-negatively graded vector space $A = \oplus_{n \geq 0} A^n$, endowed with a unital left-symmetric product $\cdot: A \otimes A \to A$, and a derivation $\partial: A \to A$, satisfying:
\begin{itemize}
\item $\1 \in A^0;$
\item $A^m \cdot A^n \subset A^{m+n};$
\item $\partial A^n \subset A^{n+1}.$
\end{itemize}
Throughout the paper, we will assume all $A^n$ to be finite-dimensional vector spaces.

\begin{ex}
Let $A = \Cset[x]$, and set $\partial = x^2 d/dx$. If we choose $x$ to have degree $1$, then $A$ is a differential graded commutative algebra, hence also a DGLsA.
\end{ex}

\subsection{Lie conformal algebras}

A {\em Lie conformal algebra} is a $\Cset[\partial]$-module $L$ endowed with a $\lambda$-bracket
$$R \otimes R \ni a \otimes b \mapsto [a_\lambda b] \in R[\lambda]$$
satisfying
\begin{itemize}
\item $[\partial a_\lambda b] = - \lambda [a_\lambda b], \qquad [a_\lambda \partial b] = (\partial + \lambda)[a_\lambda b];$
\item $[a_\lambda b] = - [b_{-\partial - \lambda} a];$
\item $[a_\lambda[b_\mu c]] - [b_\mu [a_\lambda c]] = [[a_\lambda b]_{\lambda+ \mu} c],$
\end{itemize}
whenever $a, b, c \in R$. Lie conformal algebras have been introduced in \cite{K} and studied in \cite{DK} in order to investigate algebraic properties of local families of formal distributions. This notion, and its multi-variable generalizations, are deeply related to linearly compact infinite-dimensional Lie algebras and their representation theory.

\subsection{Vertex algebras}

Let $V$ be a complex vector space. A {\em field} on $V$ is a formal power series $\phi(z)
\in (\End V)[[ z, z^{-1}]]$ with the property that $\phi(z)v \in
V((z)) = V[[z]][z^{-1}],$ for every $v \in V$. In other words, if
$$\phi(z) = \sum_{i \in \Z} \phi_i z^{-i-1}$$ then $\phi_n(v) = 0 $
for sufficiently large $n$.

A {\em vertex algebra} is a (complex) vector space $V$ endowed with a
linear {\em state-field correspondence} $Y:V \to (\End V)[[z,
z^{-1}]]$, a {\em vacuum element} $\1$ and a linear {\em (infinitesimal) translation operator}
$\partial \in \End V$ satisfying the following properties:
\begin{itemize}
\item $Y(v,z)$ is a field for all $v\in V$. \hfill {\em (field axiom)}
\item For every $a, b \in V$ one has
$$(z-w)^N[Y(a,z), Y(b,w)] = 0$$ for sufficiently large $N$. \hfill {\em (locality)}
\item The vacuum element $\1$ is such that
$$Y(\1,z) = \id_V,\qquad Y(a,z)\1 \equiv a \mod
zV[[z]],$$ for all $a \in V$. \hfill {\em (vacuum axiom)}
\item $\partial$ satisfies $$[\partial, Y(a,z)]
= Y(\partial a, z) = \frac{d}{dz}Y(a,z),$$ for all $a\in V$. \hfill {\em (translation invariance)}
\end{itemize}

One usually writes
$$Y(a,z) = \sum_{j\in \Z} a_{(j)} z^{-j-1}.$$
and views the $\C$-bilinear maps $a \otimes b \mapsto
a_{(j)} b, j \in \Z,$ as products describing the vertex algebra
structure. The normally ordered product $ab = \,:\!\!ab\!\!:\,= a_{(-1)} b$ determines all negatively labeled
products as
$$j!\, a_{(-j-1)} b = (\partial^j a)_{(-1)} b.$$
Non-negatively labeled products can
be grouped in a generating series
$$[a _\lambda b] = \sum_{n \geq 0} \frac{\lambda^n}{n!} a_{(n)} b,$$
which can be showed to define a Lie conformal algebra structure.
The compatibility conditions between the normally ordered product and the $\lambda$-bracket are well understood \cite{BK,DsK}, and amount to imposing quasi-commutativity
\begin{equation}\label{quasicomm}
[a,b] = \int_{-\partial}^0 d\lambda\, [a_\lambda b] ,
\end{equation}
and the noncommutative Wick formula
\begin{equation}\label{wick}
[a_\lambda bc] = [a_\lambda b]c + b [a_\lambda c] + \int_0^\lambda d\mu\, [[a_\lambda b]_\mu c].
\end{equation}

As a consequence, the normally ordered product may fail to be associative. The associator $(a, b, c) := (ab)c - a(bc)$ can be expressed in the form
\begin{equation}\label{assoc}
(a, b, c) = \left( \int_0^\partial d\lambda\, a\right) [b_\lambda c] + \left( \int_0^\partial d\lambda\, b\right) [a_\lambda c],
\end{equation}
hence it satisfies $(a, b, c) = (b, a, c)$. $V$ is therefore a left-symmetric algebra with respect to its normally ordered product. Because of \eqref{commass} and \eqref{quasicomm}, one obtains commutativity and associativity of the normally ordered product as soon as the $\lambda$-bracket vanishes. The operator $\partial$ is a derivation of all products. As the normally ordered product is non-associative, we will denote by $:a_1 a_2 \dots a_n:$ the product $a_1(a_2(\dots (a_{n-1}a_n)\dots))$ obtained by associating on the right.

\subsection{Vertex operator algebras}

In this paper, a {\em vertex operator algebra} (henceforth, a {\em VOA}) is a non-negatively graded vector space $V = \bigoplus_{n \geq 0} V^n$, endowed with a vertex algebra structure such that
\begin{itemize}
\item
The normally ordered product and translation operator $\partial$ make $V$ into a DGLsA;
\item
$\Tor(V) = V^0 = \Cset \1$;
\item
There exists a {\em Virasoro element} --- i.e., an element $\omega \in V^2$ satisfying $$[\omega_\lambda \omega] = (\partial + 2\lambda)\omega + \frac{c}{12}\lambda^3\1,$$ for some $c \in \Cset$ --- such that
$[\omega_\lambda a] = (\partial + n\lambda)a + O(\lambda^2)$, for all $a \in V^n$.
\end{itemize}
As a consequence, $V^i \,_{(n)} V^j \subset V^{i + j - n - 1}$, $\partial V^i \subset V^{i+1}$. By $\Tor V$, I mean the torsion of $V$ when viewed as a $\Cset[\partial]$-module.

\section{Interaction between normally ordered product and $\lambda$-bracket}
As the structure of a vertex algebra is described by the normally ordered product, along with the $\lambda$-bracket, it is interesting to figure out how much each of the two products determines the other.

\subsection{The normally ordered product of a VOA determines the $\lambda$-bracket} 
\label{determines}
We know that the $\lambda$-bracket of a vertex algebra $V$ is polynomial in $\lambda$, and determines the commutator of elements as in \eqref{quasicomm}. If we choose elements $c_j \in V$ so that
$$[a_\lambda b] = \sum_{j=0}^n \lambda^j c_j,$$
then we may compute
$$[\partial^i a, b] = (-1)^i \cdot \sum_{j=0}^n \int_{-\partial}^0  \lambda^{i+j}c_j d\lambda = \sum_{j=0}^n (-1)^j \frac{\partial^{i+j+1}c_j}{i+j+1},$$
hence
\begin{equation}\label{hilbert}
\partial^{n-i}[\partial^ia, b] = \sum_{j=0}^n \frac{(-1)^j}{i+j+1} \cdot \partial^{n+j+1} c_j.
\end{equation}

As soon as we are knowledgeable about the normally ordered product of the vertex algebra $V$, we are able to compute the left-hand side of \eqref{hilbert} for every $i = 0, \dots, n$; as coefficients of the right-hand sides form a non-degenerate matrix, we can then solve \eqref{hilbert} as a system of linear equations, and recover uniquely the values of $\partial^{n+j+1}c_j, j=0, \dots, n$. In other words, the normally ordered product determines each coefficient $c_j$ up to terms killed by $\partial^{n+j+1}$.\\

We have already seen that every VOA is a DGLsA with respect to its normally ordered product. 
\begin{thm}
A DGLsA structure may be lifted to a VOA structure in at most one way. 
\end{thm}
\begin{proof}
It is enough to show that the normally ordered product uniquely determines the $\lambda$-bracket. Let $a\in V^h, b \in V^k$. Then $[a_\lambda b]$ is a polynomial in $\lambda$ of degree at most $n=h+k-1$. Proceeding as above, we may determine all of its coefficients up to terms killed by some power of $\partial$. However, $\Tor V = \Cset \1$, so $[a_\lambda b]$ is determined up to multiples of $\1$.
By \eqref{assoc}, we have
$$(a, u, b) = \left( \int_0^\partial d\lambda\, a\right) [u_\lambda b] + \left( \int_0^\partial d\lambda\, u\right) [a_\lambda b].$$
If we choose $u$ so that $u, a$ are $\Cset[\partial]$-linearly independent, we may now determine unknown central terms in $[a_\lambda b]$.

Such a choice of $u$ is always possible, as we may assume without loss of generality that $a\notin \Tor V$, otherwise $[a_\lambda b] = 0$ \cite{DK}; we may also assume that $V$ has rank at least two, otherwise, if $a$ is non-torsion, unknown central terms in $[a_\lambda a]$ can  be computed using 
$$(a, a, a) = 2 \left( \int_0^\partial d\lambda\, a\right) [a_\lambda a].$$
The value of $[a_\lambda a]$ now uniquely determines the Lie conformal algebra structure.
\end{proof}

\subsection{The $\lambda$-bracket determines vertex algebra ideals}

If $A$ and $B$ are subsets of $V$, define products
\begin{eqnarray*}
A \cdot B\,\, & = & \spn_\Cset \langle a_{(n)}b \,|\, a\in A, b \in B, n \in \Zset\},\\
\llbr A,B\rrbr & = & \spn_\Cset \langle a_{(n)}b \,|\, a\in A, b \in B, n\geq 0\}.
\end{eqnarray*}
If $B$ is a $\CD$-submodule of $V$, then $A \cdot B, \llbr A, B\rrbr$ are also $\CD$-submodules. If $A, B$ are both $\CD$-submodules of $V$, then $A \cdot B = B \cdot A, \llbr A, B\rrbr = \llbr B, A\rrbr$. A $\CD$-submodule $I \subset V$ is a {\em vertex algebra ideal} if $I \cdot V \subset I$; it is a {\em Lie conformal algebra ideal} if $\llbr I, V \rrbr \subset I$. An element $a \in V$ is {\em central} if $\llbr a, V\rrbr = 0$.
\begin{lemma}[\cite{simple, varna}]\label{kac}
If $B, C \subset V$ are $\CD$-submodules, then $\llbr A, B\rrbr \cdot C \subset \llbr A, B \cdot C\rrbr$. In particular, if $X$ is a subset of $V$, then $\llbr X, V\rrbr$ is an ideal of $V$.
\end{lemma}
This observation has an immediate drawback: every vertex algebra $V$ is in particular a Lie conformal algebra. If $I$ is an ideal of this Lie conformal algebra structure, then $J = \llbr I, V\rrbr \subset I$ is an ideal of the vertex algebra $V$, which is certainly contained in $I$. The induced $\lambda$-bracket on the quotient $I/J$ is trivial. We may rephrase this by saying that every Lie conformal algebra ideal of $V$ sits centrally on a vertex algebra ideal. In the case of a VOA, a stronger statement holds:
\begin{thm}
Let $V$ be a VOA. 
A subspace $\1 \notin I\subset V$ is an ideal for the vertex algebra structure if and only if it is an ideal of the underlying Lie conformal algebra.
\end{thm}
\begin{proof}
The grading of $V$ is induced by the Virasoro element $\omega$. As $\llbr I, \omega \rrbr \subset \llbr I, V\rrbr \subset I$, $I$ must contain all homogeneous components of each of its elements. However, if $a \in V$ is a homogeneous element (of nonzero degree), then $a \in \llbr a, \omega\rrbr$. This forces $I$ to equal $\llbr I, V\rrbr$, which is a vertex algebra ideal.
\end{proof}
\begin{rem}
Notice that $\Cset \1$ is always a Lie conformal algebra ideal of $V$, but is never an ideal of the vertex operator algebra structure.
\end{rem}

\subsection{Different notions of ideal in a vertex algebra}

A vertex algebra structure is made up of many ingredients, that may stand by themselves to provide meaningful concepts. In particular, a vertex algebra is naturally endowed with both a (differential) left-symmetric product, and a Lie conformal algebra structure, and we may consider ideals with respect to each of the above structures. To sum it up, we have
\begin{itemize}
\item {\em Vertex algebra ideals}: ideals of the vertex algebra structure --- closed under $\partial$, $:ab:$, $[a_\lambda b]$;
\item {\em Lie conformal algebra ideals:} ideals of the Lie conformal algebra structure --- closed under $\partial, [a_\lambda b]$;
\item {\em DLs ideals:} ideals of the differential left-symmetric structure --- closed under $\partial$, $:ab:$.
\end{itemize}
When $V$ is a VOA, we have seen that the first two notions (more or less) coincide. In what follows, we will mostly be concerned with simple VOAs, i.e., VOAs with no nontrivial vertex ideals. Notice that even a simple VOA does possess many DLs ideals. Both the normally ordered product and the differential $\partial$ increase the grading, so that if $a \in V^h$, then the DLs ideal generated by $a$ is contained in $\oplus_{n \geq h} V^n$.

We conclude that the only nontrivial concept in a simple VOA is that of DLs ideal; thus, the term {\em ideal} will henceforth refer to DLs ideals alone. Notice that we may distinguish between left, right and two-sided ideals, whereas vertex algebra and Lie conformal algebra ideals are always two-sided.

\section{Finiteness of VOAs}

\subsection{Strong generators of a VOA}

When dealing with finiteness of vertex algebras, the notion that has naturally emerged in the (both mathematical and physical) literature depends only on the (differential) left-symmetric algebra structure.
A vertex algebra $V$ is called {\em strongly finitely generated} if there exists a finite set of generators such that normally ordered products of derivatives of the generators $\Cset$-linearly span $V$; this is equivalent to being able to choose finitely many quantum fields so that every element of $V$ can be obtained from the vacuum state by applying a suitable polynomial expression in the corresponding creation operators. This definition makes no reference whatsoever to the $\lambda$-bracket; when dealing with finiteness phenomena it is natural to only resort to concepts that are independent of the Lie conformal algebra structure.

\subsection{Hilbert's Basissatz and the fundamental theorem of invariant theory}

If $A = \oplus_{n\geq 0} A^n$ is a finitely generated commutative associative unital graded algebra, and $G$ is a reductive group acting on $A$ by graded automorphisms, then the subalgebra $A^G$ of $G$-invariants is also finitely generated. Hilbert's celebrated proof of this fact uses noetherianity of $A$ in an essential way: if $I$ is the ideal of $A$ generated by the positive degree part $A^G_+$ of $A^G$, then any finite subset of $A^G_+$ generating $I$ as an ideal is also a finite set of generators of $A^G$ as an algebra.

\subsection{Does the orbifold construction preserve finiteness of a VOA?}

It is natural to ask whether Hilbert's strategy can be extended to the wider setting of VOAs. Indeed, the mathematical and physical literature provide scattered example of strongly finitely generated (simple) VOAs for which the invariant subalgebra relative to the action of a reductive group of graded automorphisms stays strongly finitely generated. However, no general argument is known that applies to all examples.

A major difficulty in understanding the general algebraic aspect of the above phenomena depend on its failure in commutative examples. We have seen that a commutative VOA is nothing but a differential commutative associative algebra. However, it is not difficult to provide examples of differentially finitely generated commutative associative algebras whose invariant part with respect to the action of a finite group of graded automorphisms does not stay finitely generated.

The strongly finite generatedness of invariant subalgebra does therefore depend on noncommutative quantum features, and any attempt to provide a general address must address the problem of understanding why the commutative case behaves so differently.

\subsection{Failure of noetherianity in the differential commutative setting and non-finiteness of invariant subalgebras}

Consider the commutative ring $A = \Cset[u^{(n)}, n \in \Nset ]$ of polynomials in the countably many indeterminates $u^{(n)}$. Setting $\partial u^{(n)} = u^{(n+1)}$ uniquely extends to a derivation of $A$, thus making it into a differential commutative algebra.

Consider now the unique differential automorphism $\sigma$ of $A$ satisfying $\sigma(u) = -u$. Then clearly $\sigma(u^{(n)}) = -u^{(n)}$ and $\sigma(u^{(n_1)} \dots u^{(n_h)}) = (-1)^h u^{(n_1)} \dots u^{(n_h)}.$ It is not difficult to see that $A^{\langle \sigma \rangle} = \Cset[u^{(i)} u^{(j)}, i, j \in \Nset]$. However, $A^{\langle \sigma \rangle}$ admits no finite set of differential algebra generators.

\begin{rem}
If we endow $A$ with a trivial $\lambda$-bracket, then $A$ is an example of a commutative vertex algebra. Notice that setting $\deg(u^{(n)}) = n+1$ provides $A$ with a grading compatible with the vertex algebra structure. However, $A$ is not a VOA as there is no Virasoro element inducing this grading.
\end{rem}

It is easy to adapt Hilbert's argument to the differential commutative setting {\em once} noetherianity is established. An inevitable consequence of the above counterexample is that the differential commutative algebra $A$ must fail to satisfy the ascending chain condition on differential ideals. This fact has long been known \cite{Ritt}, and effort has been put into providing some weaker statement replacing and generalizing noetherianity. We recall the following classical result:
\begin{thm}[Ritt]
Let $A$ be finitely generated as a differential commutative $K$-algebra, where $K$ is a field of characteristic zero. Then $A$ satisfies the ascending chain condition on its {\bf radical} differential ideals.
\end{thm}
In Ritt's language, radical differential ideals are {\em perfect}, and generators of a perfect ideal as an ideal (resp. as a perfect ideal) are called strong (resp. weak) generators. The above statement claims that all perfect ideals have a finite set of weak generators, but they may well fail to have a finite set of strong generators.

Under a different meaning of weak vs. strong generators, this difference of finiteness property shows up again in the context of VOAs.

\section{An abelianizing filtration for VOAs}

The problem of finding strong generators for a VOA can be addressed by using a decreasing abelianizing filtration introduced\footnote{Li's setting is more general than ours, as the grading is only assumed to be bounded from below.} in \cite{abelian}. We recall here (a slight variant of) its definition and some of its main properties. In what follows, if $X, Y \subset A$ are subsets, we will set $AB = \spn_\Cset \langle ab|a \in A, b \in B\rangle$. Notice that $AB \neq A \cdot B$, in general.

\subsection{Li's filtration}

If $A$ is a DGLsA, set $E_n(A), n \in \Nset$ to be the linear span of all products (with respect to all possible parenthesizations) $$\partial^{d_1} a_1 \, \partial^{d_2} a_2 \, \dots \, \partial^{d_h} a_h,$$
where $a_i \in A$ are homogeneous elements, and $d_1 + \dots d_h \geq n$. Also set $E_n(A) = A$ if $n<0$. The $E_i(A), i \in \Nset$ form a decreasing sequence
$$A = E_0(A) \supset E_1(A) \supset \dots \supset E_n(A) \supset \dots$$
of subspaces of $A$, and clearly satisfy
\begin{eqnarray}
E_i(A)E_j(A)&\subset& E_{i+j}(A);\\\label{prod}
\partial \,E_i(A)&\subset& E_{i+1}(A).\label{der}
\end{eqnarray}
In particular, each $E_i(A)$ is an ideal of $A$. If $a \in E_i(A) \setminus E_{i+1}(A)$, then we will say that $a$ has rank $i$, and will denote by $[a]$ the element $a + E_{i+1}(A) \in E_i(A)/E_{i+1}(A)$.
\begin{lemma}\label{grok}
If $V$ is a VOA, then $[E_i(V), E_j(V)] \subset E_{i+j+1}(V)$ for all $i, j$.
\end{lemma}
\begin{proof}
Follows immediately from \eqref{quasicomm}.
\end{proof}

\begin{prop}\label{commgrad}
Let $V$ be a VOA. Then $[a][b] = [ab], \partial[a] = [\partial a]$ make
$$\gr V = \oplus_{i \geq 0} E_i(V)/E_{i+1}(V)$$
into a graded commutative (associative) differential algebra.
\end{prop}
\begin{proof}
Well definedness of the product is clear. Its commutativity follows from Lemma \ref{grok}.
By \eqref{commass}, associativity follows from commutativity and left-symmetry of the product in $V$. Finally, $\partial$ is well-defined, and its derivation property descends to the quotient.
\end{proof}
\begin{rem}
Li proves that, if $V$ is a VOA, then $\gr V$ can be endowed with a Poisson vertex algebra structure. However, we will not need this fact. 
\end{rem}

\begin{thm}[Li]\label{gen}
Let $X$ be a subset of homogeneous elements of a VOA $V$. Then $X$ strongly generates $V$ if and only if elements $[x], x \in X,$ generate $\gr V$ as a differential commutative algebra.
\end{thm}

In other words, a VOA $V$ is strongly finitely generated if and only if $\gr V$ is finitely generated as a differential commutative algebra.

\subsection{Strong generators of ideals}

The problem of finding strong generators for a VOA is closely connected to that of finding nice sets of generators for its ideals.

Recall that, if $A$ is a DGLsA, $I\subset A$ is a {\em (two-sided, right) ideal} of $A$ if it is a (two-sided, right) homogeneous differential ideal. We denote by $(X))$ the smallest right ideal of $A$ containing a given subset $X \subset A$, and similarly, by $((X))$, the smallest two-sided ideal containing $X$. 
A subspace $U \subset A$ is {\em strongly generated} by $X \subset U$ if $U = (\CD X)A $. When dealing with strongly generated ideals, we will henceforth abuse notation and write $XA$ for $(\CD X)A$.

We rephrase another of Li's results as follows
\begin{thm}\label{genideal}
Let $I$  be a right ideal of a VOA $V$. Then $\gr I$ is a (differential) ideal of $\gr V$, and $X\subset V$ strongly generates $I$ if and only if $[x], x \in X,$ generate $\gr I$ as a differential ideal of $\gr V$.
\end{thm}
We can easily apply this statement to elements of Li's filtration.
\begin{prop}\label{dd}
Let $X$ be a set of homogeneous generators of a VOA $V$. Then $E_d(V)$ is strongly generated by monomials
$$:(\partial^{d_1} x_{1}) \, (\partial^{d_2} x_{2}) \, \dots \, (\partial^{d_{h-1}} x_{h-1}) \, (\partial^{d_h} x_{h}):,$$
where $x_i \in X$, and $d_i>0$ satisfy $d_1 + \dots + d_h = d$. In particular, if $V$ is finitely generated, then $E_d(V)$ is a strongly finitely generated ideal.
\end{prop}
\begin{proof}
It follows immediately by noticing that $E_n(V)/E_{n+1}(V)$ is linearly generated by classes of monomials
$$:(\partial^{d_1} x_{1}) \, (\partial^{d_2} x_{2}) \, \dots \, (\partial^{d_{h-1}} x_{h-1}) \, (\partial^{d_h} x_{h}):,$$
where $x_i \in X$, and $d_i \geq 0$ satisfy $d_1 + \dots + d_h = n$.
\end{proof}

\subsection{Weak vertex algebras}

In order to construct and use Li's filtration, we do not need the full power of VOAs. Indeed, the $E_i(A)$ always constitute a decreasing filtration of the DGLsA $A$ and satisfy \eqref{prod}, \eqref{der}. In order to show that $\gr A$ is commutative and associative, we also need 
\begin{equation}\label{bra}
[E_i(A), E_j(A)] \subset E_{i+j+1}(A).
\end{equation}
This certainly holds in VOAs, but stays true under weaker conditions.

\begin{defn}
A {\em weak VOA} is a DGLsA $A = \oplus_{i \geq 0} A^i$ satisfying \eqref{bra}.
\end{defn}
\begin{ex}
\qquad\
\begin{itemize}
\item Every non-negatively graded differential commutative (associative) algebra is a weak vertex operator algebra.
\item Every VOA is a weak vertex operator algebra.
\item Let $V$ be a VOA, $I \subset V$ a two-sided ideal. Then $V/I$ is a weak vertex operator algebra: indeed, $V/I$ is a DGLsA and constructing Li's filtration commutes with the canonical projection. Notice that $V/I$ fails to be a VOA, unless $I$ is a vertex algebra ideal.
\end{itemize}

\end{ex}
If $A$ is a weak VOA, then $\partial^d A \subset \oplus_{i \geq d} A^i$, hence $E_n(A) \subset \oplus_{i \geq n} A^i$. Thus, $\cap_n E_n(A) = (0)$, and $E_i(A) \cap A^j = (0)$ as soon as $i>j$. Proposition \ref{commgrad}, \ref{dd} and Theorems \ref{gen}, \ref{genideal} easily generalize to the weak VOA setting.

Chains of inclusions between ideals in a weak VOA also behave nicely, due to the following observation:
\begin{lemma}
Let $I\subset J$ be right ideals of a weak VOA $A$ satisfying $\gr I = \gr J$. Then $I = J$.
\end{lemma}
\begin{proof}
If $X \subset I$ generates $\gr I$ as an ideal of $\gr A$, then it also generates $\gr J$, hence $I = J = XA$.
\end{proof}

\section{The ascending chain condition in a VOA}

\subsection{Full ideals}

\begin{defn}
Let $I$ be a right ideal of a VOA $V$. Then $I$ is {\em full} if $E_N(V) \subset I$ for sufficiently large values of $N$.
\end{defn}
Full ideals are important because of the following key observation.
\begin{thm}\label{fg}
Let $V$ be a strongly finitely generated VOA, $I \subset V$ a full right ideal.
Then $I$ is a strongly finitely generated ideal.
\end{thm}
\begin{proof}
As $I$ is full, it contains $E_N(V)$ for some $N \geq 0$. Then $\bar I = I/E_N(V)$ is an ideal of the quotient weak VOA $\bar V = V/E_N(V)$.

Notice that if $u_1, \dots, u_n$ are (strong) generators of $V$, then $\bar u_1, \dots, \bar u_n$ generate $\bar V$, hence elements $[\bar u_i]$ generate $\gr \bar V$ as a differential commutative associative algebra. However, only finitely many derivatives of each $[\bar u_i]$ are nonzero.
Therefore, $\gr \bar V$ is a finitely generated, and not just differentially finitely generated, commutative algebra. By Hilbert's basis theorem, the ideal $\gr \bar I$ is finitely generated, and we may apply the weak VOA version of Theorem \ref{genideal} to show that $I$ is strongly finitely generated modulo some $E_N(V)$. However, Proposition \ref{dd} shows that all ideals $E_N(V)$ are strongly finitely generated, hence $I$ is so too.
\end{proof}

By using a variant of the argument in Section \ref{determines}, one is able to prove the following statement.
\begin{lemma}\label{full}
Let $I$ be a right ideal of the VOA $V$. Then $I$ is full as soon asany one of the following properties is satisfied
\begin{itemize}
\item $I$ is nonzero and $V$ is a simple VOA;
\item $I$ contains some derivative of the Virasoro element $\omega$, provided that the central charge is nonzero;
\item $I$ is two-sided, and contains some derivative of the Virasoro element $\omega$.
\end{itemize}
\end{lemma}

\subsection{Noetherianity}

\begin{prop}
Let $V$ be a finitely generated VOA. Then $V$ satisfies the ascending chain condition on its full right ideals.
\end{prop}
\begin{proof}
If
$$I_1 \subset I_2 \subset \dots \subset I_n \subset I_{n+1} \subset \dots$$
is an ascending sequence of full right ideals, set $I = \cup_n I_n$. Then $I$ is a full ideal, and we may use Theorem \ref{fg} to locate a finite $X \subset I$ such that $I = XV$. Due to finiteness of $X$, one may find $N \geq 0$ such that $X \subset I_N$. Then $I = XV \subset I_N$.
\end{proof}
All the following statements are now of immediate proof.
\begin{thm}\label{noetherian}
Every simple VOA satisfies the ascending chain condition on its right ideals.
\end{thm}
\begin{thm}
Let $V$ be a VOA, $X \subset V$ a subset containing $\partial^i \omega$ for some $i \geq 0$. Then there exists a finite subset $X_0 \subset X$ such that $((X)) = ((X_0))$.
\end{thm}
\begin{thm}
Let $V$ be a simple VOA, $X \subset V$. Then there exists a finite subset $X_0 \subset X$ such that $(X)) = (X_0))$.
\end{thm}

We may rephrase Theorem \ref{noetherian} by saying that every simple finitely generated VOA is right-noetherian.

\begin{rem}
Notice that, unless $V$ is associative (e.g., when $V$ is commutative), subspaces of the form $XV$ may fail to be right ideals, so the above reasoning {\bf does not} prove that if
$$X_1 \subset X_2 \subset \dots \subset X_n \subset X_{n+1} \subset \dots$$
is an increasing family of subsets, then the corresponding sequence
$$X_1 V \subset X_2 V \subset \dots \subset X_n V \subset X_{n+1} V\subset \dots$$
stabilizes. In other words, we do not know whether a simple fintiely generated VOA must satisfy the ascending chain condition {\bf also} on its subspaces of the form $XV$.
\end{rem}

\begin{rem}
Finite generation of every right ideal $I$ in a simple finitely generated VOA $V$ is a strong claim. However, one often needs a stronger statements which may easily fail.

Say that $I = (X))$ or even $I = XV$. Then it is true that one may find a finite subset $X_0 \subset I$ such that $I = X_0 V$, but there is no clear way to force $X_0 \subset X$.
The standard proof of this fact would require the ascending chain condition in the stronger form stated above.
\end{rem}
\section{Speculations on Hilbert's approach to finiteness in the\\
 VOA orbifold setting}

\subsection{Subspaces of the form $XV$}

Let $a, b$ be elements of a VOA $V$. Then \eqref{assoc} shows that $(a, b, c)\in aV + bV$ for every choice of $c \in V$. However, $(a, b, c) = (ab)c - a(bc)$; as $a(bc) \in aV$, then $(ab)c \in aV + bV$ for all $c \in V$. We can summarize this in the following statement:
\begin{lemma}
Let $V$ be a VOA, $X \subset V$ a collection of homogeneous elements not containing $\1$. Then $XV = \langle X \rangle_+ V$.
\end{lemma}
\begin{proof}
It is enough to show that if $u$ is a product of (derivatives) of elements from $X$, then $uV \subset XV$. This follows from the previous lemma and an easy induction on the number of terms in the product.
\end{proof}
\begin{prop}
Let $U \subset V$ be VOAs, $X \subset V$ a collection of homogeneous elements not containing $\1$. Then
$$X \mbox{ strongly generates } U \implies U_+ \subset XV \implies U_+ \subset XV + VX.$$
\end{prop}
The above implications can be reversed for certain classes of subalgebras.

\subsection{Split subalgebras}

Let $U \subset V$ be VOAs.
\begin{defn}
$U$ is a {\em split subalgebra} of $V$ if there exists a graded $\CD$-submodule decomposition $V = U \oplus M$ such that $UM \subset M$.
\end{defn}
Whenever $U$ is a split subalgebra of $V$, there exists a $\CD$-linear splitting $R: V \to U$ which is a homomorphism of $U$-modules. The splitting clearly satisfies $R^2 = R$, and
$R(uv) = uR(v), R(vu) = vR(u)$ for every $u\in U, v\in V$.

\begin{ex}
If $G$ is a reductive group acting on the finitely generated VOA $V$ by graded automorphisms, then $V^G$ is a split subalgebra of $V$.
\end{ex}

\begin{thm}\label{Hilbert}
Let $U$ be a split subalgebra of the VOA $V$, $X \subset U$ a collection of homogeneous elements not containing $\1$. Then
$$U_+ \subset XV + VX \implies X \mbox{ strongly generates } U.$$
\end{thm}
\begin{proof}
Let $u \in U$ be a homogeneous element of positive degree. As $u \in U_+ \subset XV + VX$ there exist finitely many nonzero elements $r_x^i, s_x^i \in V$, that we may assume homogeneous without loss of generality, such that
$$u = \sum_{x\in X, i\in \Nset} r_x^i \partial^i x + \partial^i x s_x^i.$$
As $R(u) = u$, then also
$$u = \sum_{x \in X, i \in \Nset} R(r_x^i) \partial^i x + \partial^i x R(s_x^i).$$
In order to show that $u$ can be expressed as a linear combination of products of elements from $X$, it is enough to notice that $R(r_x^i), R(s_x^i)$ are homogeneous elements from $U$ of lesser degree than $u$, and proceed by induction on the degree.
\end{proof}

\subsection{(Not quite) proving that the VOA orbifold construction preserves finiteness}

Let $V$ be a simple finitely generated VOA, $G$ a reductive group acting on $V$ by automorphisms. Then both the following statements hold:
\begin{itemize}
\item $(V^G_+)) = (U))$ for some finite set $U \subset V^G_+$;
\item $(V^G_+)) = XV$ for some finite set $X \subset (V^G_+))$.
\end{itemize}
We are however not able to show none of the following increasingly weaker statements
\begin{itemize}
\item $(V^G_+)) = XV$ for some finite set $X \subset V^G_+$,
\item $V^G_+ V = XV$ for some finite set $X \subset V^G_+$,
\item $V^G_+ V + V V^G_+ = XV + VX$ for some finite set $X \subset V^G_+$,
\end{itemize}
which would suffice to apply Theorem \ref{Hilbert} to ensure finiteness of $V^G$. 
Such statements depend on a stronger Noetherianity property than we are able to show.

Notice that the above proof of right Noetherianity of a simple finitely generated VOA requires considering nonzero associators, thus resulting in a strictly noncommutative statement. Noncommutative VOAs are however typically nonassociative, and this may prevent subspaces of the form $XV$ from being right ideals.

It is not clear how one should proceed to adapt Hilbert's strategy to the VOA setting. I would like to list a few (bad and good) facts one must necessarily cope with.
\begin{itemize}
\item $XV$ can fail to be an ideal of $V$.
\item Furthermore, it is easy to construct examples of $X \subset V$ such that $\gr XV$ is not an ideal of $V$. The ideal property is likely to fail for subspaces $\gr (XV + VX)$ too.
\end{itemize}
However the proof of many statements does not require the full strenght of ideals:
\begin{itemize}
\item $A \subset B, \gr A = \gr B \implies A = B$ holds for subspaces, not just ideals.
\item If $\1 \in \llbr a, b\rrbr$, then $aV + Vb$ contains some $\partial^N V$. However this does not seem to guarantee fullness.
\item If $X \subset V$ is non-empty, then $(XV)V$ may fail to be an ideal, but is however full.
\item If $A \subset V$ is a subspace such that $\gr A$ contains $\gr E_n(V)$, then $A$ contains $E_n(V)$.
\end{itemize}
It is also possible that strong finite generation of subspaces of the form $XV$ may fail in general, but can be proved in the special case of $X = V^G_+$.\\

\noindent{\bf Problem:} understand what conditions ensure that a subspace $XV + VX$ contain a nonzero ideal.

\end{document}